\documentclass[reqno]{amsart}
\usepackage{amscd}
\usepackage[dvips]{graphics}

\def\beq{\begin{equation}}
\def\eeq{\end{equation}}
\def\ba{\begin{array}}
\def\ea{\end{array}}

\def\cal{\mathcal}

\numberwithin{equation}{section}
\newenvironment{abs}{\textbf{Abstract}\mbox{  }}{ }
\newenvironment{key words}{\textbf{Keywords}\mbox{  }}{ }
\newtheorem{theorem}{Theorem}[section]

\newtheorem{corollary}[theorem]{\textbf{Corollary}}
\newtheorem{proposition}[theorem]{\textbf{Proposition}}
\newtheorem{lemma}[theorem]{Lemma}
\renewenvironment{proof}{\noindent{\textbf{Proof.}}}{\hfill$\Box$}
\theoremstyle{remark}

\theoremstyle{plain}

\begin{document}
\title{\textbf{Liouville theorems on the upper half space}}
\author  {Lei Wang  and Meijun Zhu}
\address{ Lei Wang, Academy of Mathematics and Systems Sciences, Chinese Academy of Sciences, Beijing 100190, P.R. China, and Department of Mathematics, The University of Oklahoma, Norman, OK 73019, USA}

\email{wanglei@amss.ac.cn}

\address{ Meijun Zhu, Department of Mathematics,
The University of Oklahoma, Norman, OK 73019, USA}

\email{mzhu@math.ou.edu}

\maketitle

\noindent
\begin{abs}
In this paper we shall establish some Liouville theorems for solutions bounded from below to certain linear elliptic equations on the upper half space. In particular, we show that  for $a \in (0, 1)$  constants are the only $C^1$ up to the boundary positive solutions to $div(x_n^a \nabla u)=0$ on the upper half space.
\end{abs}\\

\section{\textbf{Introduction}\label{Section 1}}









In this paper we shall establish some Liouville theorems for solutions bounded from below to certain linear elliptic equations on the upper half space. These results imply  the uniqueness property to  various extension operators on the upper half space. They also provide us a new view point on how to obtain  positive kernels for the extension operators. The elliptic properties and estimates, as well as the geometric applications of these extension operators were  widely studied recently, see, for example, Caffarelli and Silvestre \cite{CS07}, Hang, Wang and Yang \cite{HWY08},  Chen \cite{Chen14}, Dou and Zhu \cite{DZ15},  Dou, Guo and Zhu \cite{DGZ17}, Gluck \cite{G18}, and references therein.

\subsection{Main results} Denote $\mathbb{R}^n_+=\{x=(x', x_n) \in \mathbb{R}^n \ : \  x_n>0\}$ as the upper half space. 
%
%
We  shall prove

\begin{theorem}\label{main-1}
	For $n\geq2$ and $a\in\mathbb{R}$, let  $u(x) \in C^2(\mathbb{R}^n_{+})\cap C^0(\overline {\mathbb{R}^n_{+}})$ be a solution to
	\begin{equation}\label{equ}
	\begin{cases}
	div(x_{n}^{a} \nabla u)=0, & \quad u>-C  \quad \text{in}\quad \mathbb{R}^n_+,\\
	u=0, &\quad\text{on}\quad \partial\mathbb{R}^n_+.
	\end{cases}
	\end{equation}
Then $u=C_*x_n^{1-a}$  for some nonnegative constant $C_*$;
         
	\end{theorem}



\smallskip

For Neumann boundary condition, we have

\begin{theorem}\label{main-2}
	Assume $n\geq2$ and $\max\{-1,2-n\}< a<1$. Suppose $u(x) \in C^2(\mathbb{R}^n_{+})\cap C^1(\overline {\mathbb{R}^n_{+}})$ satisfies 
	\begin{equation}\label{equ-1}
	\begin{cases}
	div(x_{n}^{a} \nabla u)=0, &\quad u>0, \quad \text{in}\quad \mathbb{R}^n_+,\\
	x_n^a\frac {\partial u}{\partial x_n}=0&\quad\text{on}\quad \partial\mathbb{R}^n_+.
	\end{cases}
	\end{equation}
	Then $u=C$ for some positive constant $C$.
\end{theorem}

The boundary condition in  \eqref{equ-1} holds  in the following sense:
\begin{equation}\label{bdry-1}
\lim_{x_n \to 0^+} x_n^a\frac {\partial u}{\partial x_n}=0.
\end{equation}

Note that  for $a>0$,  if  $u(x) \in C^1(\overline {\mathbb{R}^n_{+}})$, it automatically satisfies  \eqref{bdry-1}. We immediately have the following result.
\begin{corollary}\label{str-1}
Assume $n\ge2$ and $0<a<1$. Suppose $u(x) \in C^2(\mathbb{R}^n_{+})\cap C^1(\overline {\mathbb{R}^n_{+}})$  satisfies 
	\begin{equation}\label{equ-1-11}
	div(x_{n}^{a} \nabla u)=0, \quad u>0, \quad \text{in}\quad \mathbb{R}^n_+.
	\end{equation}
Then $u=C$ for some positive constant $C$.
	\end{corollary}
 Corollary \ref{str-1} is quite striking: there is {\it no assumption} on the boundary value of $u(x)$. It is worth pointing out that the result in Corollary \ref{str-1} does not hold for $a=0$. And the condition $u(x) \in C^1(\overline {\mathbb{R}^n_{+}})$  can not be weakened since $u(x)=x_n^{1-a}$ does satisfy equation \eqref{equ-1-11} and is positive on the upper half space.
 
 Combining Corollary \ref{str-1} with the classical Liouville Theorem for positive harmonic functions in the whole space, we have the following generalized Liouville Theorem.
 
 \begin{corollary}\label{str-2}
Assume $n\ge2$ and $0\le a<1$. Any positive $C^2(\mathbb{R}^n) $ solution to  
	\begin{equation}\label{equ-1-12}
	div(|x_{n}|^{a} \nabla u)=0, \quad \text{in}\quad \mathbb{R}^n
	\end{equation}
must be a constant function.
	\end{corollary}

%
\smallskip

We illustrate some motivations for our work below.

\subsection{Unique solution to  the extension operators} In  \cite{CS07}, Caffarelli and Silvestre study the following extension problem for $a \in (-1, 1)$:
\begin{equation}\label{equ1.2-1}
	\begin{cases}
	div(x_{n}^{a} \nabla u)=0,  & \quad \text{in}\quad \mathbb{R}^n_+,\\
	u(x', 0)=f(x'), & \quad\text{on}\quad \partial\mathbb{R}^n_+.
	\end{cases}
	\end{equation}
Besides many interesting properties were obtained, their study provides a nice ``pointwise'' view on a global defined fractional Laplacian operator:
$$(-\Delta)^{\frac{1-a}2}f(x')=- \lim_{x_n \to 0^+} x_n^a\frac {\partial u}{\partial x_n}(x', x_n).$$	

For $f(x')$ in a good space (for example, Fourier transform can be applied on $f(x')$), solution $u(x', x_n)$ to \eqref{equ1.2-1} can be represented, up to a constant multiplier, by
\begin{equation}\label{equ1.2-2}
u(x', x_n)= \int_{\partial \mathbb{R}_+^n} \frac{ x_n^{{1-a}} f(y)}{(|x'-y|^2+x_n^2)^{\frac{n-a}2}}dy.
\end{equation}
One can also view $u(x', x_n)$ as an extension of $f(x')$ via operator ${\cal P}_a$:
$$
u(x', x_n)= {\cal P}_a(f) :=\int_{\partial \mathbb{R}_+^n}  P_a(x'-y, x_n) f(y)dy,
$$
whose positive kernel is
\begin{equation}\label{equ1.2-3}
P_a(x', x_n)=\frac{ x_n^{{1-a}}}{(|x'|^2+x_n^2)^{\frac{n-a}2}}.
\end{equation}

Hang, Wang and Yan \cite{HWY08} obtain the sharp $L^p$ estimates on ${\cal P}_0$ for $n \ge 3$ (the standard harmonic extension with Poisson kernel). Their results were generalized by Chen \cite{Chen14} for general $a > 2-n$. Note that for $n=2$, from Chen's result one can obtain a different proof of two dimensional  analytic isoperimetric inequality for simply connected domains due to Carleman \cite{Cal}. 


Quite naturally, one may ask: are there other solutions to \eqref{equ1.2-1} besides the function given in \eqref{equ1.2-2}? Generally, the answer is yes, since there are many sign-changing solutions to \eqref{equ}.  However, if one only considers bounded solutions, our Theorem \ref{main-1} indicates that the  function given in \eqref{equ1.2-2} is the only one.

\smallskip

To extend the classical Hardy-Littlewood-Sobolev inequality on the upper half space, Dou and Zhu \cite{DZ15} studied the following extension operator for $\alpha \in (1, n):$
\begin{equation}\label{equ1.2-4}
u(x', x_n)= {\cal E}_\alpha  (f) :=\int_{\partial \mathbb{R}^n_+}  E_\alpha (x'-y, x_n) f(y)dy:= \int_{\partial \mathbb{R}^n_+}\frac{f(y)}{(|x'-y|^2+x_n^2)^{\frac{n-\alpha}2}}dy.
\end{equation}
The sharp $L^p$ estimates were obtained in \cite{DZ15}. Later, more general extension operators on the upper half space were studied by Dou, Guo and Zhu \cite{DGZ17} and Gluck \cite{G18}.

Direct computation shows that $u(x', x_n)= {\cal E}_{2-a}f$, up to some constant multiplier,  satisfies
\begin{equation}\label{equ1.2-5}
	\begin{cases}
	div(x_{n}^{a} \nabla u)=0,  & \quad \text{in}\quad \mathbb{R}^n_+,\\
	x_n^a \frac{\partial u}{\partial x_n}=f(x')&\quad\text{on}\quad \partial\mathbb{R}^n_+.
	\end{cases}
	\end{equation}
Theorem \ref{main-2} indicates that for $a \in (\max\{2-n,-1\}, 1)$ the bounded solution to \eqref{equ1.2-5} is unique.

\subsection{New view point on the positive kernels} The classical way to find the fundamental solution to Laplacian operator on $\mathbb{R}^n$ is to solve an ordinary differential equation, by assuming that the solution is radially symmetric. This approach certainly fails if the domain is the upper half space.

The other  view point to find the fundamental solution could be like this. First, the constant solution $u=C$ is the only positive harmonic solutions in the whole space $\mathbb{R}^n$ (for simplicity, let us just consider  $n \ge 3$). Its kelvin transformation: $v(x)=\frac 1{|x|^{n-2}}u(\frac x{|x|^2})$, which is a positive harmonic function on $\mathbb{R}^n \setminus \{0\}$, will yield the fundamental solution (up to a constant multiplier).

To find the positive kernel for the equation \eqref{equ1.2-1} and \eqref{equ1.2-5}, we first have the following observation, which will be proved in next section.

\smallskip
\noindent{\bf Lemma \ref{invariant} } If  $u(x)\in C^2(\mathbb{R}_+^n)$ satisfies the equation $div(x_{n}^{a} \nabla u)=0$ on $\mathbb{R}_+^n$, then $v(x)=\frac 1{|x|^{n-2+a}}u(\frac x{|x|^2})$ satisfies the same equation.
\smallskip

Combining Lemma \ref{invariant} with Theorem \ref{main-1} and \ref{main-2}, we know that the kernel for  the equation \eqref{equ1.2-1} and \eqref{equ1.2-5} for $a\ne 2-n$, up to a constant multiplier,  are given by
$$
\Gamma_d=\frac {x_n^{1-a}}{|x|^{n-a}},  \quad \text{and} \quad \  \Gamma_n=\frac 1{|x|^{n-2+a}}
$$
respectively. 

In \cite{DGZ17}  Dou, Guo and Zhu studied a general extension operator using a kernel obtained by taking a partial derivative of Riesz kernel along $x_n$ direction. Later, Gluck \cite{G18} studied a more general extension operator ${\cal E}_{\alpha, \beta}$ with the positive kernel
$$
E_{\alpha, \beta}(x', x_n)= \frac{ x_n^{{\beta}}}{(|x'|^2+x_n^2)^{\frac{n-\alpha}2}}$$
for $\beta\ge0, \  0<\alpha+\beta<n-\beta.$
Notice that
$$x_n^b E_{a, 1-a-b}(x', x_n)=\Gamma_d(x', x_n).$$
So all these in \cite{DGZ17} and \cite{G18} are really not  ``new'' positive kernels.

\subsection{Discussion} 
We point out that: for $a=0$, Theorem \ref{main-1} seems to be a folklore for nonnegative harmonic functions. We do not know the original proof for this fact.  One way to prove it is to adapt the approach by  Gidas and Spuck  in \cite{GS82}. Unfortunately, It seems to us that their approach only works for nonnegative functions and for $a=0$.  Here, we use the method of moving sphere, introduced by Li and Zhu in \cite{LZ95}. Note that we only assume that $u(x)$ is bounded from below in Theorem \ref{main-1}.

 For $a=0$, Theorem \ref{main-2} (after we make an even reflection of the solutions)  follows from the classical Liouville theorem in the whole space: the only positive harmonic functions in $\mathbb{R}^n$ are positive constants.
It seems that Theorem \ref{main-2} is still true for $a \notin (-1, 1)$.  But our method does not work. 

It is also interesting to extend Corollary 1.3 to other unbounded domains.

\section{Invariance}

For any fixed $x\in\partial\mathbb{R}^n_+$ and $\lambda>0$, we define
$$y^{x,\lambda}=x+\frac{\lambda^2(y-x)}{|y-x|^2}, \quad \forall y\in \overline{ \mathbb{R}^n_+},$$
and
  $$u_{x,\lambda}(y)=\frac{\lambda^{n-2+a}}{|y-x|^{n-2+a}}u(y^{x, \lambda}), \quad \forall y \in \overline{\mathbb{R}^n_+}.$$

We have the  following invariant property.
\begin{lemma}\label{invariant}
	If  $u(y)\in C^2(\mathbb{R}_+^n)$ satisfies equation $div(y_{n}^{a} \nabla u)=0$ on $\mathbb{R}_+^n$, then for any $x\in \partial \mathbb{R}^n_+$ and $\lambda>0$, $u_{x, \lambda}(y)$ satisfies the same equation. 
		\end{lemma}
	
\begin{proof}
	 By a direct computation, we have for $i=1,2,\cdots,n-1,$
	 
	 \begin{equation}\label{gi}
	 \begin{split}
	 	\partial_i u_{x,\lambda}(y)=&-\frac{(n-2+a)\lambda^{n-2+a}(y_i-x_i)}{|y-x|^{n+a}}u(y^{x,\lambda})
	 	+\frac{\lambda^{n+a}}{|y-x|^{n+a}}\partial_i u(y^{x,\lambda})\\
	 	&-\frac{2\lambda^{n+a}(y_i-x_i)}{|y-x|^{n+2+a}}\nabla u(y^{x,\lambda})\cdot(y-x),
	 	\end{split}
	 \end{equation}
for $i=n$, 
\begin{equation}\label{n}
\begin{split}
	\partial_n u_{x,\lambda}(y)=&-\frac{(n-2+a)\lambda^{n-2+a}y_n}{|y-x|^{n+a}}u(y^{x,\lambda})
	+\frac{\lambda^{n+a}}{|y-x|^{n+a}}\partial_n u(y^{x,\lambda})\\
&-\frac{2\lambda^{n+a}y_n}{|y-x|^{n+2+a}}\nabla u(y^{x,\lambda})\cdot(y-x),
\end{split}
\end{equation}
and 
\begin{equation*}
	\begin{split}
\Delta u_{x,\lambda}(y)=&\frac{\lambda^{n+2+a}}{|y-x|^{n+2+a}}(\Delta u)(y^{x,\lambda})
+\frac{2a\lambda^{n+a}}{|y-x|^{n+a+2}}\nabla u(y^{x,\lambda})\cdot(y-x)\\
&+\frac{a(n-2+a)\lambda^{n-2+a}}{|y-x|^{n+a}}u(y^{x,\lambda}).
\end{split}
\end{equation*}
Then
\begin{equation}\label{tran}
   \begin{split}
   	div(y_n^a\nabla u_{x,\lambda})(y)&=y_n^a\Delta u_{x,\lambda}(y)+ay_n^{a-1}\partial_n u_{x,\lambda}(y)\\
   	&=\frac{\lambda^{n+2+a}y_n^a}{|y-x|^{n+2+a}}(\Delta u)(y^{x,\lambda})+a\frac{\lambda^{n+a}y_n^{a-1}}{|y-x|^{n+a}}\partial_n u(y^{x,\lambda})\\
   	&=\frac{\lambda^{n+2-a}}{|y-x|^{n+2-a}}div (y_n^a\nabla u)(y^{x,\lambda})\\
	&=0.
   \end{split}
\end{equation}
 
\end{proof}

It will be interesting to further explore the geometric implication of the above invariance.
	
\section{Dirichlet condition}
We present the proof for Theorem \ref{main-1} in this section.  Noting the specialty of $a=2-n$  in Lemma \ref{invariant},
we divide the proof into three cases: $a>2-n$,  $a< 2-n$ and $a=2-n$.  We shall prove the results using  the method of moving sphere.
\smallskip


\noindent \textbf{Case {1}.} $a>2-n$.

\smallskip

Due to  technical difficulties in dealing  with the zero boundary condition, we shall classify  all  solutions bounded from below plus a positive constant instead. It is sufficient to prove
\begin{theorem}\label{main-1-1}
Assume $n\geq2$. Suppose $u(y) \in C^2(\mathbb{R}^n_{+})\cap C^0(\overline {\mathbb{R}^n_{+}})$  satisfies
  \begin{equation}\label{1}
 \begin{cases}
  div(y_{n}^{a} \nabla u)=0, &\quad u>\frac 12 , \quad \text{in}\quad \mathbb{R}^n_+,\\
  u=1&\quad\text{on}\quad \partial\mathbb{R}^n_+.
  \end{cases}
\end{equation} 
If $a> 2-n$, then $u=C_*y_n^{1-a}+1$ for some nonnegative constant $C_*$. In particular, for $a \ge 1$, $u=1$.
\end{theorem}

From now on to the end of this section, we always assume solution $u(x) \in C^2(\mathbb{R}^n_{+})\cap C^0(\overline {\mathbb{R}^n_{+}})$.  W first have
 \begin{lemma}\label{lem1} Assume that $a> 2-n$ and $u$ satisfies conditions in Theorem \ref{main-1-1}.
    For any  $x\in\partial \mathbb{R}^n_+$, and $\lambda>0$, we have that  
      $$u_{x,\lambda}(y)\leq u(y), \ \ \ \ \ \forall  y\in\mathbb{R}^n_+\backslash B_\lambda(x).$$
 \end{lemma}

\begin{proof}
	For any fixed $x\in\partial\mathbb{R}^n_+$ and $\lambda>0$, define $$w_{x,\lambda}(y)=u(y)-u_{x,\lambda}(y).$$
	Since $n-2+a>0$, we have that
	\begin{equation}\label{infinite-1}
		\varliminf_{|y|\rightarrow\infty}w_{x,\lambda}(y)=\varliminf_{|y|\rightarrow\infty}u(y)-\lim_{|y|\rightarrow\infty}\frac{\lambda^{n-2+a}}{|y-x|^{n-2+a}}u(x+\frac{\lambda^2(y-x)}{|y-x|^2})\geq \frac{1}{2},
	\end{equation}
and
\begin{equation}\label{bd}
	w_{x,\lambda}(y)=1-\frac{\lambda^{n-2+a}}{|y-x|^{n-2+a}}>0\quad\textrm{on }\partial\mathbb{R}^n_+\backslash \overline{B_\lambda(x)}.
\end{equation}
By \eqref{infinite-1}, we know that there is an $N=N(x,\lambda)>0$ large enough, such that $w_{x,\lambda}\geq C>0$ in $\mathbb{R}^n_+\backslash \overline{B_N(x)}$. Define $\Omega=B^+_N(x)\backslash \overline{B_\lambda^+(x)}$, then we have 
	$$
	\begin{cases}
	div(y_{n}^{a} \nabla w_{x,\lambda})=0, & \quad \text{in}\quad \Omega   \\
	w_{x,\lambda} \ge 0 &\quad\text{on}\quad \partial  \Omega.
			\end{cases}
	$$
By the maximum principle, we know $w_{x,\lambda}\geq 0$ in $\Omega$. Therefore, $w_{x,\lambda}\geq 0$ in $\mathbb{R}^n_+\backslash B_{\lambda}(x)$.
\end{proof}

\smallskip

To conclude our proof, we need the following key lemma for the method of moving sphere. See, for example, the proof in Dou and Zhu \cite{DZ15}.
  \begin{lemma}\label{lem5}
 Assume $f(y) \in C^0(\overline {\mathbb{R}^n_+})$, $n\geq2$, and  $\tau>0$. If
  $$(\frac{\lambda}{|y-x|})^{\tau}f(x+\frac{\lambda^2(y-x)}{|y-x|^2})\leq f(y), \quad \forall \lambda>0,\;x\in\partial\mathbb{R}^n_+,\;|y-x|\geq\lambda,\, y\in \mathbb{R}^n_+,$$
  then $$f(y)=f(y',y_n)=f(0',y_n),\quad \forall y=(y', y_n)\in\mathbb{R}^n_+.$$
\end{lemma}

From Lemma \ref{lem1} and Lemma \ref{lem5},  we know that $u(y', y_n)=u(y_n).$   Then by solving the corresponding ODE, we obtain  $u=C_1 y_n^{1-a}+C_2$. From the boundary condition, we know:  for $2-n<a<1$, $C_2=1$; And for $a\geq 1$, $C_1$ must be $0$ and $C_2$ must be $1$. We thus complete the proof of  Theorem \ref{main-1-1}.  
\bigskip


\noindent \textbf{Case {2}.} $a < 2-n$. 

For $a <2-n$, it is easy to check that $1+y_n^{1-a}$ does not satisfy the monotonic property in Lemma \ref{lem1}, but $y_n^{1-a}-1$ does.  It is sufficient to prove


\begin{theorem}\label{main-5-1}
Assume $n\geq2$, and $a <2-n$.  If  $u(y) \in C^2(\mathbb{R}^n_{+})\cap C^0(\overline {\mathbb{R}^n_{+}})$ satisfies
	\begin{equation}\label{equ-main-5-1}
	\begin{cases}
	div(y_{n}^{a} \nabla u)=0, &\quad u\ge -2, \quad \text{in}\quad \mathbb{R}^n_+,\\
	u=-1 &\quad\text{on}\quad \partial\mathbb{R}^n_+,
	\end{cases}
	\end{equation}
   then $u=C_*y_n^{1-a}-1$  for some nonnegative constant $C_*$.
          \end{theorem}

       First, we have   
  
  \begin{lemma}\label{lem1-1} Assume that $a<2-n$ and $u$ satisfies conditions in Theorem \ref{main-5-1}.
    For any  $x\in\partial \mathbb{R}^n_+$, and $\lambda>0$, we have that  
      $$u_{x,\lambda}(y)\leq u(y), \ \ \ \ \ \forall  y\in\mathbb{R}^n_+\backslash B_\lambda(x).$$
 \end{lemma}

\begin{proof}
Similar to the proof of Lemma \ref{lem1}, 	for any fixed $x\in\partial\mathbb{R}^n_+$ and $\lambda>0$,  we define $$w_{x,\lambda}(y)=u(y)-u_{x,\lambda}(y).$$
Noting $\lim_{|y| \to 0} u(x+ y)=-1$, and  $n-2+a<0$, we have that
	\begin{equation}\label{infinite-2}
	\begin{split}
			\varliminf_{|y|\rightarrow\infty}w_{x,\lambda}(y)&=\varliminf_{|y|\rightarrow\infty}u(y)-\lim_{|y|\rightarrow\infty}\frac{\lambda^{n-2+a}}{|y-x|^{n-2+a}} u (x+\frac{\lambda^2(y-x)}{|y-x|^2}) \\
						& \ge -2+\lim_{|y|\rightarrow\infty} \frac{ \lambda^{n-2+a}}{2|y-x|^{n-2+a}}\\
			&=+\infty,
			\end{split}
	\end{equation}
and
\begin{equation}\label{bd-1}
	w_{x,\lambda}(y)=-1+\frac{\lambda^{n-2+a}}{|y-x|^{n-2+a}}>0\quad\textrm{on }\partial\mathbb{R}^n_+\backslash \overline{B_\lambda(x)}.
\end{equation}
By \eqref{infinite-2}, we know that there is an $N=N(x,\lambda)>0$ large enough, such that $w_{x,\lambda}\geq C>0$ in $\mathbb{R}^n_+\backslash \overline{B_N(x)}$. Define $\Omega=B^+_N(x)\backslash \overline{B_\lambda^+(x)}$, then we have 
	$$
	\begin{cases}
	div(y_{n}^{a} \nabla w_{x,\lambda})=0, & \quad \text{in}\quad \Omega   \\
	w_{x,\lambda} \ge 0 &\quad\text{on}\quad \partial  \Omega.
			\end{cases}
	$$
By the maximum principle, we know that  $w_{x,\lambda}\geq 0$ in $\mathbb{R}^n_+\backslash B_{\lambda}(x)$.
\end{proof}

\smallskip

Similarly, from Lemma \ref{lem1-1} and Lemma \ref{lem5},   we obtain  $u=C_* y_n^{1-a}-1$ for some nonnegative constant $C_*$, thus complete the proof of  Theorem \ref{main-5-1}.

\bigskip


\noindent \textbf{Case {3}.} $a =2-n$. 

We modify the $u_{x,\lambda}$ to be
$$u_{x,\lambda}(y)=u(y^{x,\lambda})+\ln \frac{\lambda}{|y-x|}.$$
Then $u_{x,\lambda}$ satisfies the same equation
$$div (y_n^{2-n}\nabla u)=0.$$

      
  \begin{lemma}\label{lem1-2} Assume that $u$ satisfies conditions in Theorem \ref{main-1} with $a=2-n$.
    For any  $x\in\partial \mathbb{R}^n_+$, and $\lambda>0$, we have that  
      $$u_{x,\lambda}(y)\leq u(y), \ \ \ \ \ \forall  y\in\mathbb{R}^n_+\backslash B_\lambda(x).$$
 \end{lemma}

\begin{proof}
For any fixed $x\in\partial\mathbb{R}^n_+$ and $\lambda>0$,  we define $$w_{x,\lambda}(y)=u(y)-u_{x,\lambda}(y).$$
Noting $\lim_{|y| \to 0}u(x+y)=0,$  we have that
	\begin{equation}\label{infinite-4}
	\begin{split}
			\varliminf_{|y|\rightarrow\infty}w_{x,\lambda}(y)&=\varliminf_{|y|\rightarrow\infty}u(y)-\lim_{|y|\rightarrow\infty}u (x+\frac{\lambda^2(y-x)}{|y-x|^2})-\lim_{|y|\rightarrow\infty}\ln\frac{\lambda}{|y-x|} \\
						& \ge  -C-(-\infty)\\
						&=+\infty,
			\end{split}
	\end{equation}
and
\begin{equation}\label{bd-4}
	w_{x,\lambda}(y)=\ln\frac{|y-x|}{\lambda}>0\quad\textrm{on }\partial\mathbb{R}^n_+\backslash \overline{B_\lambda(x)}.
\end{equation}
Similar to the proof of Lemma \ref{lem1-1}, we have $u_{x,\lambda}(y)\leq u(y)$.
\end{proof}

   We thus can conclude our proof from the following lemma. See, for example, the proof of Lemma 3.3 in Li and Zhu \cite{LZ95}.
   
   \begin{lemma}\label{m}  
   	Suppose that $f\in C^1(\mathbb{R}^n_+)$ satisfies, for all $x\in\partial\mathbb{R}^n_+$ and $\lambda>0$, 
   	$$f(y)\geq f(x+\frac{\lambda^2(y-x)}{|y-x|^2})+\ln\frac{\lambda}{|y-x|},\quad \forall y\in\mathbb{R}^n_+.$$
   	Then $$f(y)=f(y',y_n)=f(0',y_n),\quad \forall y=(y',y_n)\in\mathbb{R}^n_+.$$
   	\end{lemma}

\section{Neumann boundary condition}

To prove Theorem \ref{main-2}, we also consider $u>1$ by replacing $u$ with $u+1$. First, we need to verify the invariance of the boundary condition under the M\"obius  transformation.

\begin{proposition}\label{lem4-1} Assume $a>-1$, and $u(x) \in C^2(\mathbb{R}^n_{+})\cap C^1(\overline {\mathbb{R}^n_{+}})$. If $\lim_{y_n\rightarrow 0^+}y_n^a\partial_n u(y)=0$, then for all $x \in \partial \mathbb{R}^n_+$ and $\lambda>0$, $\lim_{y_n\rightarrow 0^+}y_n^a\partial_n u_{x,\lambda}(y)=0$ if $\lim_{y_n \rightarrow 0^+} y \ne x$.
\end{proposition}

\begin{proof}
By \eqref{n}, we have 
\begin{equation*}
\begin{split}
&\lim_{y_n\rightarrow 0^+}y_n^a\partial_n u_{x,\lambda}(y)\\
=&-\lim_{y_n\rightarrow 0^+}y_n^{1+a}\Big(\frac{(n-2+a)\lambda^{n-2+a}}{|y-x|^{n+a}}u(y^{x,\lambda})+\frac{2\lambda^{n+a}}{|y-x|^{n+2+a}}\nabla u(y^{x,\lambda})\cdot(y-x)\Big)\\
&+\lim_{y_n\rightarrow 0^+}\frac{\lambda^{n+a}y_n^{a}}{|y-x|^{n+a}}\partial_n u(y^{x,\lambda}).
\end{split}
\end{equation*}
For $a>-1$ and $\lim_{y_n \rightarrow 0^+} y \ne x$, we have
$$\lim_{y_n\rightarrow 0^+}y_n^a\partial_n u_{x,\lambda}(y)=0+\lim_{y_n\rightarrow 0^+}\frac{\lambda^{n-a}}{|y-x|^{n-a}}(y_n^a\partial_{n}u)(y^{x,\lambda})=0.$$
\end{proof}

We will use the method of moving sphere again to prove Theorem \ref{main-2}.
\begin{lemma}\label{1.2-1}
 Assume that  $n\ge 2$ and $\max\{2-n,-1\}<a<1$. Suppose that $u(y) \in C^2(\mathbb{R}_+^n) \cap C^1(\overline{\mathbb{R}_+^n}) $ satisfies 
 $$
	\begin{cases}
	div(y_{n}^{a} \nabla u)=0, &\quad u>1, \quad \text{in}\quad \mathbb{R}^n_+,\\
	y_n^a\frac {\partial u}{\partial y_n}=0&\quad\text{on}\quad \partial\mathbb{R}^n_+.
	\end{cases}
	$$
	Then for any  $x\in\partial \mathbb{R}^n_+$, and $\lambda>0$, we have that  
 $$u_{x,\lambda}(y)\leq u(y), \ \ \ \ \ \forall  y\in\mathbb{R}^n_+\backslash B_\lambda(x).$$
\end{lemma}

\begin{proof}
    For any fixed $x\in\partial\mathbb{R}^n_+$ and $\lambda>0$, define $$w_{x,\lambda}(y)=u(y)-u_{x,\lambda}(y).$$
    Then 
    \begin{equation*}
    \varliminf_{|y|\rightarrow\infty}w_{x,\lambda}(y)=\varliminf_{|y|\rightarrow\infty}u(y)-\lim_{|y|\rightarrow\infty}\frac{\lambda^{n-2+a}}{|y-x|^{n-2+a}}u(x+\frac{\lambda^2(y-x)}{|y-x|^2})\geq 1,
    \end{equation*}
    Thus there is an $N=N(x,\lambda)>0$ large enough, such that $w_{x,\lambda}\geq C>0$ in $\mathbb{R}^n_+\backslash \overline{B_N(x)}$. Define $\Omega=B^+_N(x)\backslash \overline{B_\lambda^+(x)}$, then we have 
    \begin{equation}\label{w-equ-1}
    \begin{cases}
    div(y_{n}^{a} \nabla w_{x,\lambda})=0, & \quad \text{in}\quad \Omega   \\
    w_{x,\lambda} \ge 0 &\quad\text{on}\quad \partial  \Omega\cap\mathbb{R}^n_+,\\
    y_n^a\frac{\partial w_{x,\lambda}}{\partial y_n}=0&\quad \text{on}\quad \partial \Omega\cap\partial\mathbb{R}^n_+.
    \end{cases}
    \end{equation}
    
    We claim that $w_{x,\lambda}\geq 0$ in $\Omega$. Otherwise $m=\min_{\overline{\Omega}}w_{x,\lambda}<0$. By the maximum principle and boudary condition of $w_{x,\lambda}$, we know that $m=\min_{\partial \Omega\cap\partial\mathbb{R}^n_+}w_{x,\lambda}$.
     For $\epsilon>0$, we consider $$A_\epsilon(y):= w_{x,\lambda}(y)-\epsilon g(y),$$ where $g(y)=y_n^{1-a}$.  Easy to see that  $ div(y_{n}^{a} \nabla A_\epsilon)=0$ in $ \Omega$. Thus  there is a positive $\epsilon_0<\frac{-m}{N^{1-a}}$, such that,  for $\epsilon \in (0, \epsilon_0)$, \begin{equation}\label{min}
     \min_{\overline{\Omega} }A_\epsilon(y)=\min_{\partial \Omega \cap \partial \mathbb{R}_+^n} A_\epsilon(y).
     \end{equation}
     In fact, for $\epsilon\in(0,\epsilon_0)$, we have
     $$A_\epsilon|_{\partial\Omega\cap\mathbb{R}^n_+}\geq -\epsilon g|_{\partial\Omega\cap\mathbb{R}^n_+}\geq -\epsilon N^{1-a}>m=\min_{\partial\Omega\cap\partial\mathbb{R}^n_+}w_{x,\lambda}=\min_{\partial\Omega\cap\partial\mathbb{R}^n_+}A_\epsilon.$$
     We thus obtain \eqref{min} by the maximum principle.
     
     Let  $y^*\in \partial \Omega\cap\partial\mathbb{R}^n_+$ be one minimal point with $A_\epsilon(y^*)=\min_{\overline{\Omega} }A_\epsilon(y)<0$.  It follows that
    	\begin{equation}\label{neum-1}
    	\lim_{y \to y^*}\big(y_n^a \frac{\partial A_\epsilon }{\partial y_n}\big)(y)\ge 0.
    	\end{equation}		
      Thus 
    	$$\lim_{y \to y^*}\big(y_n^a\frac{\partial w_{{x, \lambda}}}{\partial y_n}\big) (y)\geq \lim_{y_n\rightarrow 0^+}\epsilon\big(y_n^a\frac{\partial  g}{\partial y_n}\big) (y)=\epsilon(1-a)>0.$$
    	It is in contradiction to the boundary condition in the equation \eqref{w-equ-1}.
\end{proof}

Now, we use  Lemma \ref{lem5} to conclude that $u$ only depends on $y_n$. By solving the corresponding ODE  we get Theorem \ref{main-2}.

\bigskip

 \noindent {\bf Acknowledgements}\\
 \noindent   L. Wang is supported by the China Scholarship Council for her study/research at the University of Oklahoma.  L. Wang   would like to thank Department of Mathematics at  the  University of Oklahoma  for its hospitality,  where this work has been done.

\end{document}